\documentclass[11pt, twoside]{article}   	
\usepackage{geometry}                		
\usepackage{graphicx}				
\usepackage{amsfonts}
\usepackage{latexsym}
\usepackage{mathrsfs}
\usepackage{amssymb}
\usepackage{amsmath}
\usepackage{amsthm}
\usepackage{indentfirst}
\usepackage{color}
\usepackage{graphicx}

\newtheorem{theorem}{\color{black}\indent Theorem}[section]
\newtheorem{lemma}{\color{black}\indent Lemma}[section]

\newtheorem{corollary}{\color{black}\indent Corollary}[section]

\title{THE IDEAL COUNTING FUNCTION IN CUBIC FIELDS
}
\author{Zhishan Yang\footnote{Email:~yangzs100@nenu.edu.cn\newline
Address: School of Mathematics and Statistics, Northeast Normal University, 
Changchun 130024, P.R. China}}

\date{}							

\begin{document}
\maketitle
\begin{abstract}
For a cubic algebraic extension $K$ of $\mathbb{Q}$, the behavior of the ideal counting function is considered in this paper. Let $a_{K}(n)$ be the number of integral ideals of the field $K$ with norm $n$.  An asymptotic formula is given for the sum
$$
\sum\limits_{n_{1}^2+n_{2}^2\leq x}a_{K}(n_{1}^2+n_{2}^2).
$$


\end{abstract}

\section{Introduction}

Let $K$ be an algebraic extension of $\mathbb{Q}$ with degree $d$. The associated Dedekind zeta function is defined by
$$
\zeta_K(s)=\sum\limits_{\mathfrak{a}}\mathfrak{N}(\mathfrak{a})^{-s},\quad \Re s>1,
$$
where the sum runs over all integral ideals in $\mathcal{O}_K$, and $\mathfrak{N}(\mathfrak{a})$ is the norm of the integral ideal $\mathfrak{a}$. 
Since the norm of an integral ideal is a positive rational integer, the Dedekind zeta function can be rewritten as an ordinary Dirichlet series
$$
\zeta_K(s)=\sum\limits_{n=1}^{\infty}a_K(n)n^{-s},\quad \Re s>1,
$$
where $a_K(n)$ counts the number of integral ideals $\mathfrak{a}$ with norm $n$ in $K$, we call it the ideal counting function. A great number of deep arithmetic properties of a number field are hidden within its Dedekind zeta function.

It is known that the ideal counting function $a_K(n)$ is a multiplicative function, and it has the upper bound
$$
a_K(n)\ll \tau^d(n),
$$
where $\tau(n)$ is the divisor function, see~\cite{Chandrasekharan K}. 

Landau~\cite{Landau} in 1927 gave the average behavior of the ideal counting function
$$
\sum\limits_{n\leq x}a_K(n)=cx+O(x^{\frac{n-1}{n+1}+\varepsilon})
$$
for arbitrary algebraic number field of degree $d\geq2$.
Nowak~\cite{Nowak} then established the best estimation hitherto in any algebraic number field of degree $d\geq3$. By using the decomposition of prime number $p$  in $\mathcal{O}_{K}$ and the analytic properties of $L$-functions, in paper~\cite{Lv}, L\"{u} considered the average behavior of  moments of the ideal function
$$
\sum\limits_{n\leq x}a_K^l(n),\quad l=1,2,\cdots
$$
and gave a sharper estimates for $l=1$ in the Galois extension over $\mathbb{Q}$, while later  L\"{u} and the author~\cite{Yang}  gave a bound for the sum
$$
\sum\limits_{n\leq x}a_K^l(n^2),\quad l=1,2,\cdots
$$
in the Galois extension over $\mathbb{Q}$.

For a non-normal extension $K$ of $\mathbb{Q}$, the decomposition of $p$ in $\mathcal{O}_{K}$ is complex, so we can not use the same method as the normal extension. 
In 2008, by applying the so-called strong Artin conjecture, Fomenko~\cite{Fo} get 
the results
$$
\sum\limits_{n\leq x}a_K^l(n),\quad l=2,\,3,
$$
when $K$ is a non-normal cubic field extension. Later, L\"{u}~\cite{Lv1} revised the error term.

In this paper, 
the author will 
be interested in the estimation of the following sum
\begin{equation}
\label{the aim}
\sum_{n_1^2+n_2^2\leq x}a_K(n_1^2+n_2^2).
\end{equation}
where $K$ is the cubic algebraic extension of $\mathbb{Q}$.

For the purpose, we first consider the arithmetic function $r(n)$ which is the number of representation of an integer $n$ as sums of two square integers. $i.e.$
$$
r(n)=\#\{n\in\mathbb{Z}| n=n_1^2+n_2^2\}.
$$
Then, we can rewrite the formula~\eqref{the aim} as
\begin{equation}
\sum_{n_1^2+n_2^2\leq x}a_K(n_1^2+n_2^2)=\sum\limits_{n\leq x}a_K(n)\sum_{n=n_1^2+n_2^2}1=\sum\limits_{n\leq x}a_K(n)r(n).
\end{equation}

It is known that $r(n)$ is the ideal counting function of the Gaussian number field $\mathbb{Q}(\sqrt{-1})$ and we have
$$
r(n)=\sum\limits_{d|n}\chi^{\prime}(d),
$$
where $\chi^{\prime}$ is a real primitive Dirichlet character modulo $4$.

For general quadratic number field $L$ with discriminant $D^{\prime}$, the ideal counting function of the field $L$ is
$$
a_{L}(n)=\sum\limits_{d|n}\chi^{\prime}(d),
$$
where $\chi^{\prime}$ is a real primitive Dirichlet character modulo $\left|D^{\prime}\right|$. It is an interesting question to consider the sum
$$
\sum\limits_{n\leq x}a_K(n)a_{L}(n).
$$
Fomenko~\cite{Fo1} consider this convolution sum when both $K$ and $L$ are quadratic field. However, we shall discuss a more general case. Assume that $q\geq1$ is an integer, $\chi$ is a primitive character modulo $q$, define the function
$$
f_\chi(n)=\sum\limits_{k|n}\chi(k),
$$
then we have the following results
\begin{theorem}
\label{theorem1}
Let $K$ be a cubic normal extension of $\mathbb{Q}$ and $q\geq1$ is an integer,  $\chi$ a primitive Dirichlet character modulo $q$, then we have
\begin{equation}
\sum\limits_{n\leq x}a_K(n)f_\chi(n)=xP_{5}(\log x)+O(x^{5/8+\varepsilon}),
\end{equation}
where $P_{5}(t)$ is a polynomial in $t$ with degree $5$, and $\varepsilon>0 $ is an arbitrarily  small constant.  
\end{theorem}

\begin{theorem}
\label{thmnonnormal}
Let $K$ be a cubic non normal extension of $\mathbb{Q}$ and $q\geq1$ is an integer,  $\chi$ a primitive Dirichlet character modulo $q$, then we have
\begin{equation}
\sum\limits_{n\leq x}a_K(n)f_\chi(n)=xP_{3}(\log x)+O(x^{5/8+\varepsilon}),
\end{equation}
where $P_{3}(t)$ is a polynomial  in $t$ with degree $3$, and $\varepsilon>0$ is an arbitrarily  small constant.  
\end{theorem}

According to the theorems above, we obtain
\begin{corollary}
Let $K$ be a cubic normal extension of $\mathbb{Q}$, and $r(n)$ the number of representation of an integer $n$ as sums of two square integers.  Then we have
$$
\sum\limits_{n\leq x}a_K(n)r(n)=xP_{5}(\log x)+O(x^{5/8+\varepsilon}),
$$
where $P_{5}(t)$ is a polynomial  in $t$ with degree $5$.
\end{corollary}

\begin{corollary}
Let $K$ be a cubic non normal extension of $\mathbb{Q}$, and $r(n)$ the number of representation of an integer $n$ as sums of two square integers.  Then we have
$$
\sum\limits_{n\leq x}a_K(n)r(n)=xP_{3}(\log x)+O(x^{5/8+\varepsilon}),
$$
where $P_{3}(t)$ is a polynomial in $t$ with degree $3$.
\end{corollary}

\bigskip
Assume that $K$ and $L$ are Galois extensions of $\mathbb{Q}$  with degree $d_{1}$, $d_{2}$, respectively. According to the theory of Artin $L$-functions, the ideal counting functions $a_{K}(n)$ and $a_{L}(n)$ can be represented by the sum of characters of the representations of $Gal(K/\mathbb{Q})$ and $Gal(L/\mathbb{Q})$, respectively.
%

\section{Preliminaries}

Let $K$ be a cubic algebraic extension of $\mathbb{Q}$, and $D=df^{2}$($d$ squarefree) its discriminant;  the Dedekind zeta function of $K$ is 
$$
\zeta_{K}(s)=\sum_{n=1}^{\infty}a_{K}(n)n^{-s},\quad \text{for} ~\Re s>1.
$$
It has the Euler product
$$
\zeta_{K}(s)=\prod_{p}\left(1+\frac{a_{K}(p)}{p^{s}}+\frac{a_{K}(p^{2})}{p^{2s}}+\cdots \right).
$$
We will give some results about Dedekind zeta function of cubic field $K$ in the following.
\begin{lemma}
\label{lemmazetaKnormal}
$K$ is a normal extension if and only if $D=f^{2}$. In this case
              $$
              \zeta_{K}(s)=\zeta(s)L(s,\varphi)L(s,\overline{\varphi}),
              $$
              where $\zeta(s)$ is the Riemann zeta function and $L(s,\varphi)$ is an ordinary 
              Dirichlet series (over $\mathbb{Q}$) corresponding to a primitive character $\varphi
              $ modulo $f$.
\end{lemma}
\begin{proof}
See the lemma in ~\cite{Wolfgang}.
\end{proof}

By using lemma~\ref{lemmazetaKnormal}, the Euler product of Riemann zeta function $\zeta(s)$ and the Dirichlet $L$-functions,  we have
\begin{lemma}
\label{aKnormal}
Assume that $a_{K}(n)$ is the ideal counting function of the cubic normal extension $K$ over $\mathbb{Q}$, we get
$$
a_{K}(n)=\sum\limits_{xy|n}\varphi(x)\overline{\varphi}(y),
$$
Here $x$ and $y$ are integers. In particular, when $n=p$ is a prime, we get
\begin{equation}
\label{aKp}
a_{K}(p)=1+\varphi(p)+\overline{\varphi}(p),
\end{equation}
where $\varphi$ is a primitive character modulo $f$.
\end{lemma}

\bigskip
Assume that $K$ is a non-normal cubic extension over $\mathbb{Q}$, which  is given by an irreducible polynomial $f(x)=x^{3}+ax^{2}+bx+c$. Let $E$ denote the normal closure of $K$ 
that is normal over $\mathbb{Q}$ with degree $6$, 
and denoted the Galois group $\text{Gal}(E/\mathbb{Q})=S_{3}$.  Firstly, we will introduce some properties about $S_{3}$(See~\cite{Frohlich}, pp. 226-227 for detailed arguments). 

The elements of $S_{3}$ fall into three conjugacy classes
$$
C_{1}:~(1);\quad C_{2}:~(1,2,3),~(3,2,1);\quad C_{3}:~(1,2),~(2,3),~(1,3).
$$
with the following three simple characters: the one dimensional characters $\psi_{1}$(the principal character) and $\psi_{2}$(the other character determined by the subgroup $C_{1}\cup 
C_{2}$), and the two dimensional character $\psi_{3}$.

Let $D$ be the discriminant of $f(x)=x^{3}+ax^{2}+bx+c$ and $K_{2}=\mathbb{Q}(\sqrt{D})$. The fields $K_{2}$ and $K$ are the intermediate extensions fixed under the subgroups $A_{3}$ and $\{(1),(1,2)\}$, respectively. The extensions $K_{2}/\mathbb{Q}$, $E/K_{2}$ and $E/K$ are abelian. The Dedekind zeta function satisfy the relations
$$
\begin{array}{llll}
      \zeta_{E}(s)&=&L_{\psi_{1}}L_{\psi_{2}}L_{\psi_{3}}^{2},\\
\zeta_{K_{2}}(s)&=&L_{\psi_{1}}L_{\psi_{2}},\\
      \zeta_{K}(s)&=&L_{\psi_{1}}L_{\psi_{3}},\\
             \zeta(s)&=&L_{\psi_{1}},
\end{array}
$$
where 
$$
L_{\psi_{2}}=L(s,\psi_{2},E/\mathbb{Q})\quad\text{and}\quad L_{\psi_{3}}=L(s,\psi_{3},E/\mathbb{Q}),
$$
and $L_{\psi_{2}}=L(s,\psi_{2},E/\mathbb{Q})$ and $L_{\psi_{3}}=L(s,\psi_{3},E/\mathbb{Q})$ are Artin $L$-functions.

Kim in \cite{Kim} proved that the the strong Artin conjecture holds true for the group $S_{3}$. By using the strong Artin conjecture, the function $L_{\psi_{3}}$ also can be interpreted in 
another way~\cite{Deligne}. Let $\rho: S_{3}\rightarrow GL_{2}(\mathbb{C})$ be the irreducible two-dimensional representation. Then $\rho$ gives rise to a cuspidal representation $\pi$ 
of $\text{GL}_{2}(\mathbb{A}_{\mathbb{Q}})$. Let
$$
L(s,\pi)=\sum_{n=1}^{\infty}M(n)n^{-s}.
$$
Below we assume that $\rho$ is odd, $i.e.$ $D<0$, then $L(s,\pi)=L(s,f)$, where $f$ is a holomorphic cusp form of weight 1 with respect to the congruence group $\Gamma_{0}(|D|)$,
$$
f(z)=\sum_{n=1}^{\infty}M(n)e^{2\pi inz}.
$$
Here as usual, $L(s,\pi)$ denotes the $L$-function of the representation $\pi$, and $L(s,f)$ denotes the Hecke $L$-function of cusp form $f$. Thus $L_{\psi_{3}}=L(s,f)$ and 
\begin{equation}
\label{zetaKL}
\zeta_{K}(s)=\zeta(s)L(s,f).
\end{equation}

The formula \eqref{zetaKL} implies that 
\begin{lemma}
\label{nonaKn}
The symbols defined as above. we have
$$
a_{K}(n)=\sum_{d|n}M(d).
$$
In particular,
$$
a_{K}(p)=1+M(p),
$$
where $p$ is a prime integer.
\end{lemma}

\bigskip
To prove the theorem, we also need some well-known estimates of the relative $L$-functions in the following.

For subconvexity bounds, we have the following well-known estimates.
\begin{lemma}
For any $\varepsilon>0$, we have
\begin{equation}
\zeta(\sigma+it)\ll_{\varepsilon}(1+|t|)^{(1/3)(1-\sigma)+\varepsilon}
\end{equation}
uniformly for $1/2\leq\sigma\leq1$, and $|t|\geq 1$.
\end{lemma}
\begin{proof}
See theorem II 3.6 in the book~\cite{Tenenbaum}.
\end{proof}

For the Dirichlet $L$-series,  By using the Phragmen-Lindel\"of principle for a strip~\cite{Iwaniec} and the estimates given by Heath-Brown~\cite{HeathB}, we have the similar results
\begin{equation}
\label{Lschi}
L(\sigma+it, \chi)\ll_{\varepsilon}(1+|t|)^{(1/3)(1-\sigma)+\varepsilon}
\end{equation}
uniformly for $1/2\leq\sigma\leq1$, and $|t|\geq 1$, where $\chi$ is a Dirichlet character modulo $q$, and $q$ is an integer.

\bigskip
For the mean values of the relative $L$-functions, we have

\begin{lemma}
\label{zetaonehalf}
For any $\varepsilon>0$, we have
\begin{equation}
\int\limits_{1}^{T}\left|\zeta(\frac{1}{2}+it)\right|^{A}\ll_{\varepsilon}T^{1+\varepsilon}
\end{equation}
uniformly for $T\geq 1$, where $A=2, 4$.
\end{lemma}

\begin{lemma}
\label{Lonehalf}
For any $\varepsilon>0$, and $q$ is an integer, $\chi$ is a character modulo $q$.  We have
\begin{equation}
\int\limits_{1}^{T}\left|L(\frac{1}{2}+it,\ \chi)\right|^{A}\ll_{\varepsilon}T^{1+\varepsilon}
\end{equation}
uniformly for $T\geq 1$, where $A=2, 4$.
\end{lemma}

For Hecke $L$-functions defined in the formula \eqref{zetaKL}, we have

\begin{lemma}
For any $\varepsilon>0$, we have
\begin{equation}
\begin{array}{cccc}
\int\limits_{1}^{T}\left|L\left(\frac{1}{2}+it,\ f\right)\right|^{2}dt&\sim& CT\log T\\
\int\limits_{1}^{T}\left|L\left(\frac{1}{2}+it,\ f\right)\right|^{6}dt&\ll    &  T^{2+\varepsilon}
\end{array}
\end{equation}
uniformly for $T\geq 1$, and the subconvexity bound
$$
L(\sigma+it,\ f)\ll_{t,\ \varepsilon}(1+|t|)^{\max \{(2/3)(1-\sigma),\ 0\}+\varepsilon}
$$
uniformly for $1/2\leq \sigma\leq 2$ and $|t|\geq 1$.
\end{lemma}
\begin{proof}
The first and third results due to Good~\cite{Good}, and the second result was proved by Jutila~\cite{Jutila}.
\end{proof}

We also have the convexity bounds for the relative $L$-functions.
\begin{lemma}
Let $L(s,\ g)$ be a Dirichlet series with Euler product of degree $m\geq 2$, which means
$$
L(s,\ g)=\sum\limits_{n=1}^{\infty}a_{g}(n)n^{-s}=\prod\limits_{p<\infty}\prod\limits_{j=1}^{m}\left(1-\frac{\alpha_{g}(p,\ j)}{p^{s}}\right),
$$
where $\alpha_{g}(p,\ j),\ j=1,\ 2,\ \cdots,\ m$ are the local parameters of $L(s,\ g)$ at prime $p$ and $a_{g}(n)\ll n^{\varepsilon}$. Assume that this series and its Euler product are absolutely convergent for $\Re (s)>1$. Also, assume that it is entire except possibly for simple poles at $s=0,\ 1$, and satisfies a functional equation of Riemann type. Then for $0\leq \sigma\leq 1$ and any $\varepsilon>0$, we have
\begin{equation}
L(\sigma+it,\ g)\ll_{g,\ \varepsilon}(1+|t|)^{(m/2)(1-\sigma)+\varepsilon},
\end{equation}
and for $T\geq 1$, we have
\begin{equation}
\label{Lg}
\int\limits_{T}^{2T}\left|L(1/2+\varepsilon+it, g)\right|^{2}dt\ll_{g,\varepsilon} T^{m/2+\varepsilon}.
\end{equation}
\end{lemma}
\begin{proof}
The first result is form the lemma 2.2 in \cite{Guangshi Lv}, and the second result is from the lemma 2.1 in \cite{Guangshi Lv}.
\end{proof}

\section{Proof of Theorems}
Assume that $K$ is a cubic extension of $\mathbb{Q}$. The Dedekind zeta function of $K$ is
$$
\zeta_{K}(s)=\sum_{n=1}^{\infty}a_K(n)n^{-s}, \quad \Re s>1.
$$
Its Euler product is 
$$
\zeta_{K}(s)=\prod_{p}\left(1+\frac{a_{K}(p)}{p^{s}}+\frac{a_{K}(p^{2})}{p^{2s}}+\cdots \right),\quad \Re s>1.
$$

Let $q$ be an integer, and $\chi$ a primitive Dirichlet character modulo $q$. Define the function
\begin{equation}
\label{fchin}
f_\chi(n)=\sum\limits_{k|n}\chi(k).
\end{equation}
It is easy to check that $f_{\chi}(m n)=f_{\chi}(m)f_{\chi}(n)$, when $(m,\,n)=1$.

Since $a_{K}(n)\ll n^{\varepsilon}$, so does $a_{K}(n)f_{\chi}(n)$. We can define an $L$-function associated to the function $a_{K}(n)f_{\chi}(n)$ in the half-plane $\Re s>1$,
\begin{equation}
L_{K,\,f_{\chi}}(s)=\sum_{n=1}^{\infty}a_{K}(n)f_{\chi}(n)n^{-s},
\end{equation}
which is absolutely convergent in this region. Both $a_{K}(n)$ and $f_{\chi}(n)$ are multiplicative function, then for $\Re s>1$, the function $L_{K,\,f_{\chi}}(s)$ can be expressed by the Euler product
$$
L_{K,\,f_{\chi}}(s)=\prod_{p}\left(1+\frac{a_{K}(p)f_{\chi}(p)}{p^{s}}+\frac{a_{K}(p^{2})f_{\chi}(p^{2})}{p^{2s}}+\cdots\right),
$$
where the product runs over all primes.

\subsection*{Proof of Theorem~\ref{theorem1}}
When $K$ is cubic normal extension, according to the formula~\eqref{aKp} and \eqref{fchin}, we get the formula
\begin{equation}
\label{aKpfchip}
a_{K}(p)f_{\chi}(p)=1+\varphi(p)+\overline{\varphi}(p)+\chi(p)+\varphi(p)\chi(p)+\overline{\varphi}(p)\chi(p),
\end{equation}
where $p$ is a nature prime number.

For $\Re s>1$, we can write 
$$
M_{K,\,f_{\chi}}(s):=\zeta(s)L(s,\,\varphi)L(s,\,\overline{\varphi})L(s,\,\chi)L(s,\,\varphi\times \chi)L(s,\,\overline{\varphi}\times \chi)
$$
as an Euler product of the form
$$
\prod_{p}\left(1+\frac{A(p)}{p^{s}}+\cdots\right),
$$
where $A(p)=1+\varphi(p)+\overline{\varphi}(p)+\chi(p)+\varphi(p)\chi(p)+\overline{\varphi}(p)\chi(p)$, and the function $L(s,\,\varphi\times \chi)$ and $L(s,\,\overline{\varphi}\times \chi)$ are the Rankin-Selberg convolution $L$-function of the Dirichlet $L$-functions $L(s,\,\varphi)$, $L(s,\,\overline{\varphi})$ with the Dirichlet $L$-functions $L(s,\,\chi)$ respectively. 

By comparing it with the Euler product of $L_{K,\,f_{\chi}}(s)$, and using the formula~\eqref{aKpfchip}, we obtain
\begin{equation}
L_{K,\,f_{\chi}}(s)=M_{K,\,f_{\chi}}(s)\cdot U_{1}(s),
\end{equation}
where $U_{1}(s)$ denotes a Dirichlet series, which is absolutely convergent for $\Re s>1/2$, and uniformly convergent for $\Re s>1/2+\varepsilon$. Therefore, the function $L_{K,\,f_{\chi}}(s)$ admits an analytic continuation into the half-plane $\sigma>1/2$, having as its only singularity a pole of order $6$ at $s=1$, because $\zeta(s)$ and each of the Dirichlet $L$-functions has a simple pole at $s=1$.

\smallskip
By using the well-known inversion formula for Dirichlet series, we obtain
$$
\sum_{n\leq x}a_{K}(n)f_{\chi}(n)=\frac{1}{2\pi i}\int_{b-iT}^{b+iT}L_{K,\,f_{\chi}}(s)\frac{x^{s}}{s}ds+O(\frac{x^{1+\varepsilon}}{T}).
$$
Where $b=1+\varepsilon$ and $1\leq T\leq x$ is a parameter to be chosen later.

Shifting the path of integration to the line $\sigma= 1/2+\varepsilon$. By using Cauchy's residue theorem, we have
\begin{eqnarray}
\label{1main}
\sum_{n\leq x}a_{K}(n)f_{\chi}(n)&=&\frac{1}{2\pi i}\left\{\int_{\frac{1}{2}+\varepsilon-iT}^{\frac{1}{2}+\varepsilon+iT}+\int_{\frac{1}{2}+\varepsilon+iT}^{b+iT}+\int_{b-iT}^{\frac{1}{2}+\varepsilon-iT}\right\}L_{K,\,f_{\chi}}(s)\frac{x^{s}}{s}ds\nonumber\\
 &  &+\text{Res}_{s=1}L_{K,\,f_{\chi}}(s)\frac{x^{s}}{s}+O(\frac{x^{1+\varepsilon}}{T})\nonumber\\
 &=&xP_{5}(\log x)+J_{1}+J_{2}+J_{3}+O(\frac{x^{1+\varepsilon}}{T})
\end{eqnarray}
Where $P_{5}(t)$ is a polynomial in $t$ with degree $5$.

Using the lemmas in section 2 about the bound for the Dirichlet series, we will estimate the $J_{i}, i=1, 2, 3$ in the following.

For $J_{1}$, we have
\begin{eqnarray}
J_{1}\ll x^{1/2+\varepsilon}+x^{1/2+\varepsilon}\int_{1}^{T}\left|M_{K,\,f_{\chi}}(1/2+\varepsilon+it)\right|t^{-1}dt
\end{eqnarray}
where we have used that $U_{1}(s)$ is absolutely convergent in the region $\Re s\geq1/2+\varepsilon$ and behaves as $O(1)$ there.

By H\"older's inequality, we have
\begin{eqnarray}
\int_{1}^{T}\left|M_{K,\,f_{\chi}}(1/2+\varepsilon+it)\right|t^{-1}dt&\ll&\log T\sup_{1\leq T_{1}\leq T}T_{1}^{-1}\cdot T_{1}^{1/6+\varepsilon}\cdot T_{1}^{1/6+\varepsilon}\times\nonumber\\
& &I_{\zeta}(T_{1})^{1/4}I_{\varphi}(T_{1})^{1/4}I_{\overline{\varphi}}(T_{1})^{1/4}I_{\chi}(T_{1})^{1/4},\nonumber
\end{eqnarray}
where we have used the formula~\eqref{Lschi}, and
\begin{eqnarray}
I_{\zeta}(T_{1})                   &:=&\int\limits_{T_{1}}^{2T_{1}}\left|\zeta(1/2+\varepsilon+it)\right|^{4}dt,\nonumber\\
I_{\varphi}(T_{1})                &:=&\int\limits_{T_{1}}^{2T_{1}}\left|L(1/2+\varepsilon+it,\varphi)\right|^{4}dt,\nonumber
\end{eqnarray}
\begin{eqnarray}
I_{\overline{\varphi}}(T_{1})&:=&\int\limits_{T_{1}}^{2T_{1}}\left|L(1/2+\varepsilon+it,\overline{\varphi})\right|^{4}dt,\nonumber\\
I_{\chi}(T_{1})                     &:=&\int\limits_{T_{1}}^{2T_{1}}\left|L(1/2+\varepsilon+it,\chi)\right|^{4}dt.\nonumber
\end{eqnarray}

Now, by using lemma~\eqref{zetaonehalf} and lemma~\eqref{Lonehalf}, we have the estimation
$$
\int_{1}^{T}\left|M_{K,\,f_{\chi}}(1/2+\varepsilon+it)\right|t^{-1}dt\ll T^{1/3+\varepsilon}.
$$
So we can deduce that
\begin{eqnarray}
\label{1J1}
J_{1}\ll x^{1/2+\varepsilon}+x^{1/2+\varepsilon}T^{1/3+\varepsilon}.
\end{eqnarray}

For $J_{2}$ and $J_{3}$, we have
\begin{eqnarray}
\label{1J23}
J_{2}+J_{3}&\ll&\sup_{1/2+\varepsilon\leq \sigma\leq 1+\varepsilon}x^{\sigma}T^{-1}\left|M_{K,\,f_{\chi}}(\sigma+iT)\right|\nonumber\\
                  &\ll&\sup_{1/2+\varepsilon\leq \sigma\leq 1+\varepsilon}x^{\sigma}T^{-1}T^{(1/3+1/3+1/3+1/3+1/3+1/3)(1-\sigma)+\varepsilon}\nonumber\\
                  &\ll&\frac{x^{1+\varepsilon}}{T}+x^{1/2+\varepsilon}T^{\varepsilon}.
\end{eqnarray}

Form formula~\eqref{1main}, \eqref{1J1} and \eqref{1J23}, we have
\begin{equation}
\label{1last}
\sum_{n\leq x}a_{K}(n)f_{\chi}(n)=xP_{5}(\log x)+O(x^{1/2+\varepsilon}T^{1/3+\varepsilon})+O(\frac{x^{1+\varepsilon}}{T}).
\end{equation}

Taking $T=x^{3/8+\varepsilon}$ in \eqref{1last}, we have
$$
\sum_{n\leq x}a_{K}(n)f_{\chi}(n)=xP_{5}(\log x)+O(x^{5/8+\varepsilon}).
$$
We complete the proof of Theorem~\ref{theorem1}.

\subsection*{Proof of Theorem~\ref{thmnonnormal}}

Now, assume that $K$ is a cubic non-normal extension over $\mathbb{Q}$. According to the lemma~\ref{nonaKn} and the formula~\eqref{fchin}, we have
\begin{equation}
\label{nonaKpfchip}
a_{K}(p)f_{\chi}(p)=1+\chi(p)+M(p)+\chi(p)M(p),
\end{equation}
where $p$ is a nature prime number.

By virtue of \eqref{nonaKpfchip},  we have the relation
$$
L_{K,\,f_{\chi}}(s)=\zeta(s)L(s, \chi)L(s, f)L(s, f\times \chi)\cdot U_{2}(s),
$$
where $L(s, f\times \chi)$ is the Rankin-Selberg convolution $L$-function of $L(s,\, f)$ and $L(s,\,\chi)$, and $U_{2}(s)$ denotes a Dirichlet series, which is absolutely convergent for $\sigma>1/2$. Therefore, the function $L_{K,\,f_{\chi}}(s)$ admits an analytic continuation into the half-plane $\sigma>1/2$, having as its only singularity a pole of order $4$ at $s=1$, because $\zeta(s)$ and each of the relative $L$-functions has a simple pole at $s=1$.

The degree of $L(s,\ f\times \chi)$ is $2$, according to the formula~\eqref{Lg}, we have
$$
\int\limits_{T}^{2T}\left|L(1/2+\varepsilon+it, f\times\chi)\right|^{2}dt\ll_{g,\varepsilon} T^{2/2+\varepsilon}.
$$

Similarly as the proof of Theorem \ref{theorem1}, using the inversion formula for Dirichlet series and the estimates above, 
we have the main term of the sum is 
$$
\text{Res}_{s=1}{L_{K,\,f_{\chi}}(s)x^{s}s^{-1}}=xP_{3}(\log x),
$$
and the error term is $O(x^{5/8+\varepsilon})$.

The proof is over.



\end{document}